\documentclass[graybox]{svmult}

\usepackage{type1cm}        
\usepackage{makeidx}         
\usepackage{graphicx}        
\usepackage{multicol}        
\usepackage[bottom]{footmisc}

\usepackage{newtxtext}       %
\usepackage[varvw]{newtxmath}       

\makeindex             

\newcommand{\NN}{{\mathbb N}}

\newcommand{\RR}{{\mathbb R}}

\newcommand{\im}{\mathop{\rm im}\nolimits}

\newcommand{\rank}{\mathop{\rm rank}\nolimits}

\newcommand{\sortname}[1]{} 

\newcommand{\cP}{\mathcal{P}}
\newcommand{\cG}{\mathcal{G}}

\newcommand{\cM}{\mathcal{M}}
\newcommand{\fX}{\mathfrak{X}}
\newcommand{\eps}{\varepsilon}
\newcommand{\spann}{\mathop{\rm span}\nolimits}

\begin{document}

\title*{Overlap Splines and Meshless Finite Difference Methods} 
\author{Oleg Davydov}
\institute{Oleg Davydov \at 
University of Giessen,
Department of Mathematics,
Arndtstrasse 2,
35392 Giessen,
Germany,
\email{oleg.davydov@math.uni-giessen.de}}
%

\maketitle

\abstract{We consider overlap splines that are defined by connecting the patches of piecewise functions via common values at
given finite sets of nodes, without using any partitions of the computational domain.  It is shown that some classical finite
difference methods may be interpreted as collocation with overlap splines. Moreover, several versions of the meshless finite
difference methods,  such as the RBF-FD method, are equivalent to the collocation or discrete least squares with appropriately 
chosen spaces of overlap splines.}

\section{Introduction}
\label{intro}

In this paper we consider spaces of piecewise defined multivariate functions (for example piecewise polynomials) without
choosing a partition of the domain. Therefore the pieces overlap each other, and the "overlap spline" is a set of patches
rather than a well defined function. We connect the patches by requiring that they coincide at a finite number of points. 
In particular, it is not even expected that any underlying partition exists such that the patches form a continuous function. 
The absence of a partition eliminates the need for meshing algorithms in the numerical implementation of the overlap splines.
Despite the simplicity of this setting, basic questions such as determining the dimension of the spaces of
overlap splines do not seem trivial in general.
In the univariate case overlap splines are related to the well known discrete splines \cite{SchumSplines07}. 

In this paper we mainly concentrate on the simplest case of "interpolatory" overlap splines, and their applications to solving operator
equations by collocation and discrete least squares methods. Remarkably, it turns out that classical finite difference methods
can be interpreted as collocation with overlap splines. 
However, even more interesting is the connection to modern meshless finite difference methods, 
in particular RBF-FD methods \cite{FFprimer15},
that have grown in popularity in recent years due to their true meshless and isogeometric nature and excellent performance in
numerous numerical experiments. This is where the absence of partitions plays a decisive role. We hope that recasting
meshless finite difference methods in terms of overlap splines will contribute to developing
theoretical justification of these methods. 
Note that our recent work \cite{D23} in fact makes use of an overlap spline construction in 
its proof of the error bound for a discrete least squares version of RBF-FD. On the other hand, we hope that the overlap
spline perspective will help to develop new, improved versions of the meshless finite difference methods.

The paper is organized as follows. In Section~\ref{os} we introduce overlap splines, address
the dimension question and provide several examples, including discrete splines, multivariate polynomial and kernel based
patches. Section~\ref{pum} briefly discusses the application of overlap splines to the direct approximation of functions,
whereas Section~\ref{col} is devoted to the solution of operator equations and demonstrates how collocation and discrete least
squares with overlap splines lead to several known versions of the meshless finite difference method, including the RBF-FD method
and its oversampled variants.

\section{Overlap Splines}
\label{os}

For a finite set $M$ let $|M|$ denote its cardinality. For a vector $a\in\RR^N$ and any set $J\subset\{1,\ldots,N\}$  we
denote by $a|_{J}$ the vector $[a_j]_{j\in J}$.

Let $G$ be a set and $X=\{x_j\}_{j=1}^N$ its finite subset. Consider a collection 
$\cG=\{G_i\}_{i=1}^m$ of subsets of $G$ such that $G=\bigcup_{i=1}^m G_i$, and
a collection $\cP=\{P_i\}_{i=1}^m$ of finite dimensional vector spaces $P_i$ of real-valued 
functions defined on $G_i$, $i=1,\ldots,m$. We set $X_i=\{x_j:j\in J_i\}=X\cap G_i$ and  $n_i=|X_i|=|J_i|$.

Note that in practice we may prefer to first choose the sets $X_i$ and only thereafter find suitable enclosing sets $G_i$,  or
even just show the existence of $G_i$ with certain desired properties. For example, $G_i$ may be only needed in the proofs, whereas the generation of $X_i$ be part
of a computational algorithm. A possible approach to the latter when $G$ is a metric space is to choose  centers
$c_1,\ldots,c_m\in G$, and generate $X_i$ by selecting for each $i=1,\ldots,m$ a set of nearest neighbors  of  $c_i$ in
$X$. This can be done for example by collecting either all neighbors within certain distance from $c_i$ (range search), or a
prescribed number of its nearest neighbors (KNN search).

\begin{definition}[Overlap splines]\label{osp}
Given $X$, $\cG$ and $\cP$, any collection $s=\{p_i\}_{i=1}^m$ of \textit{patches} $p_i\in P_i$, $i=1,\ldots,m$, 
is said to be an \textit{overlap spline} if the \textit{connection condition} $p_i|_{X_i\cap X_j}=p_j|_{X_i\cap X_j}$ 
holds whenever $X_i\cap X_j\ne\emptyset$.
The linear space of all overlap splines is denoted $S(X,\cG,\cP)$. 
\end{definition}

Clearly, $\dim S(X,\cG,\cP)\le \sum_{i=1}^m\dim P_i$, and the equality holds if and only if $P_i|_{X_i\cap X_j}=\{0\}$ for all
pairs  $(i,j)$ such that $i\ne j$ and $X_i\cap X_j\ne\emptyset$. Moreover, it is easy to see that
\begin{equation}\label{lbdim}
\dim S(X,\cG,\cP)\ge |X|+\sum_{i=1}^m\dim P_i-\sum_{i=1}^m|X_i|.
\end{equation}
Indeed, we may represent each $s=\{p_i\}_{i=1}^m\in S(X,\cG,\cP)$ as a point $\RR^D$, where $D=\sum_{i=1}^m\dim P_i$, 
by expanding each $p_i$ in a basis of $P_i$, and this defines a linear mapping $S(X,\cG,\cP)\to \RR^D$. 
Each point $x_k\in X$ generates at most $m_k-1$ linearly independent linear equations for the images of $s$ in $\RR^D$ resulting from the conditions
$p_i(x_k)=p_j(x_k)$ for all $i\ne j$ such that $x_k\in X_i\cap X_j$, where $m_k:=|\{i:x_k\in X_i\}|$. 
These linear equations define the linear subspace of $\RR^D$ formed by the images of $s$, and
it follows that the dimension of this subspace is at least $D-\sum_{k=1}^N(m_k-1)$. 
Since $\sum_{k=1}^Nm_k=\sum_{i=1}^m|X_i|$, the lower bound \eqref{lbdim} follows.

We can say more about the dimension of $S(X,\cG,\cP)$ under additional assumptions. 
A set $Y\subset G$ is said to be \textit{total} for a finite-dimensional
vector space $P$ of functions on $G$  if for any $p\in P$ the condition $p|_Y=0$ implies $p=0$. This means that $p\in P$ is
completely determined by its values at $Y$. If in addition  $|Y|=\dim P$, then  $Y$ is an \textit{interpolation set (I-set)}
for  $P$, that is the interpolation problem that prescribes arbitrary values of $p$ at the points in $Y$ is uniquely solvable.
In other words, $Y$ is an I-set for  $P$ if and only if $|Y|=\dim P|_Y=\dim P$.

\begin{definition}[Interpolatory overlap splines]\label{iosp}
A space $S(X,\cG,\cP)$ of overlap splines is called \textit{interpolatory} if each $X_i$ is an I-set for $P_i$, 
$i=1,\ldots,m$.
\end{definition}

\begin{proposition}\label{dimIS}
Assume that $X_i$ is total for $P_i$, for all $i=1,\ldots,m$. Then $\dim S(X,\cG,\cP)\le|X|$.
For interpolatory overlap splines $\dim S(X,\cG,\cP)=|X|$.
\end{proposition}

\begin{proof} Given $s\in S(X,\cG,\cP)$ we define $\hat u\in \RR^N$ by setting $\hat u_j=p_i(x_j)$ for any $i$ such that
$x_j\in X_i$. It follows from Definition~\ref{osp} that $p_i(x_j)$ does not depend on a particular choice of
$i$. The mapping $T:S(X,\cG,\cP)\to \RR^N$ given by $s\mapsto\hat u$ is linear.  If $Ts=0$,
then in particular $p_i|_{X_i}=\hat u|_{J_i}=0$ and hence $p_i=0$ for all $i$, which shows that $\ker T=\{0\}$. Hence $T$ is
an injection, and $\dim S(X,\cG,\cP)\le|X|$.
This in particular applies to the interpolatory overlap splines. Moreover, in this case $|X_i|=\dim P_i$, $i=1,\ldots,m$, and
hence \eqref{lbdim}  gives the opposite inequality $\dim S(X,\cG,\cP)\ge|X|$.
\end{proof}

The mapping $T:S(X,\cG,\cP)\to \RR^N$ used in the above proof is well defined for any overlap splines. We will use
the notation $s|_X:=Ts$, and call $s|_X$ the \textit{restriction} of $s$ to $X$. Thus, any $s\in S(X,\cG,\cP)$ gives rise to a
discretized function $s|_X$ defined only on $X$. Under the hypotheses of Proposition~\ref{dimIS} we have $\ker T=\{0\}$, and hence 
each $s\in S(X,\cG,\cP)$ is uniquely determined by $s|_X$.
 
We may always split the task of computing the
dimension of $S(X,\cG,\cP)$ into the dimensions of the kernel and the image of $T$,
\begin{equation}\label{kerimT}
\dim S(X,\cG,\cP)=\dim \ker T+\dim\im T.
\end{equation}

For the spaces of interpolatory overlap splines the mapping $T$ is a bijection. Indeed, we have shown in the proof of Proposition~\ref{dimIS} that 
$\ker T=\{0\}$ in this case.  To show the
surjectivity of $T$, consider any $\hat u\in \RR^N$, and for each $i=1,\ldots,m$ choose $p_i$ to be the unique element
of $P_i$ such that $p_i(x_j)=\hat u_j$ for all $j\in J_i$. Then $s=\{p_i\}_{i=1}^m$ satisfies $Ts=\hat u$. This means that
the elements $s\in S(X,\cG,\cP)$ can be parameterized by their `values' $s(x_j)$, $j=1,\ldots,N$, uniquely defined by 
evaluating at $x_j$ any one of the suitable patches $p_i$ such that $x_j\in X_i$.

Any space $S(X,\cG,\cP)$ of interpolatory overlap splines possesses a \textit{Lagrange basis} $\ell_1,\ldots,\ell_N$, with
$\ell_j=\{\ell_{ji}\}_{i=1}^m$, where each $\ell_{ji}\in P_i$ is uniquely determined by the interpolation conditions
$\ell_{ji}(x_k)=\delta_{jk}$, $k\in X_i$, and $\delta_{jk}$ is the Kronecker delta. These basis splines are  \textit{local} in
the sense that $\ell_{ji}=0$ as soon as $x_j\notin X_i$.

We now discuss several examples of overlap splines. In all  special cases considered in this paper $G$ is a closed set in a
smooth manifold, mostly $G\subset \RR^d$. 
However we prefer not to restrict the definitions to this situation, because less usual sets $G$ may quite naturally be of
interest, e.g.\ graphs.

\subsubsection*{Univariate Generalized Splines and Discrete Splines}

Let $G$ be an interval $[a,b]\subset\RR$. Consider a partition $\Delta=\{a=t_0<t_1<\cdots<t_m=b\}$ of $[a,b]$, and choose
sets  $G_i\supset[t_{i-1},t_i]$, and spaces $P_i$, $i=1,\ldots,m$. Given a set $X=\{x_j\}_{j=1}^N\subset [a,b]$, any overlap
spline $s=\{p_i\}_{i=1}^m\in S(X,\cG,\cP)$ may be used to construct a \textit{generalized spline} $\tilde s:[a,b]\to\RR$ in
the sense of  \cite[Chapter 11]{SchumSplines07}, where $\tilde s(x):=p_i(x)$, with $i$ satisfying $x\in [t_{i-1},t_i]$ and
$x<t_i$ when $i<m$. The linear functionals of the set $\Gamma_{ij}$ of \cite[Definition 11.1]{SchumSplines07} are  the
point evaluation functionals for all points in $X_i\cap X_j$. 

The spline $\tilde s$ is a continuous function on $[a,b]$ if
each $p_i$ is continuous on $[t_{i-1},t_i]$ and  $t_i\in X$ for all $i=1,\ldots,m-1$. 
If $X_i\subset [t_{i-1},t_i]$, $i=1,\ldots,m$, then the generalized spline $\tilde s$ may be
directly recognized as an overlap spline in $S(X,\tilde\cG,\cP)$, with $\tilde\cG=\{[t_{i-1},t_i]:i=1,\ldots,m\}$.
In particular, we obtain the classical spaces of continuous piecewise polynomials (finite elements) of degree $q$ if
$X_i$ is a set of $q+1$ points in $[t_{i-1},t_i]$, including the endpoints $t_{i-1},t_i$, and $P_i$ is the linear space
$\Pi^1_q$ of univariate polynomials of degree at most $q$.

However, an overlap spline does not require a partition. According to Definition~\ref{osp} we only need to choose  points
$x_j\in [a,b]$,  sets $G_i\subset[a,b]$ and  spaces $P_i$. As an example, we take 
\begin{align}\label{ex1d}
\begin{split}
x_j&=a+hj,\quad j=0,\ldots,N,\text{ with }h=(b-a)/N,\\
G_i&=[x_{i-1},x_{i+1}],\quad i=1,\ldots,N-1, \text{ and}\\
P_i&=\Pi^1_2,\text{ the space of quadratic polynomials}.
\end{split}
\end{align}

Since $X_i=X\cap G_i=\{x_{i-1},x_i,x_{i+1}\}$, $i=1,\ldots,N-1$, are I-sets for $\Pi^1_2$,
the space $S(X,\cG,\cP)$ is interpolatory. The Lagrange basis $\ell_0,\ldots,\ell_N$ for it, with
$\ell_j=\{\ell_{ji}\}_{i=1}^{N-1}$, is given by Lagrange polynomials
\begin{equation}\label{ex1dL}
\ell_{ji}(x)=\frac{1}{h^2}
\begin{cases}
\frac12(x-x_{j-2})(x-x_{j-1}),& i=j-1,\\
-(x-x_{j-1})(x-x_{j+1}),& i=j,\\
\frac12(x-x_{j+1})(x-x_{j+2}),& i=j+1,\\
0,& \text{otherwise.}
\end{cases}
\end{equation}

If we now choose the following partition, in order to remove the overlaps, 
$\Delta=\{a=t_0<t_1<\cdots<t_{N-1}=b-h\}$ with simply $t_j=x_j$ for all $j$, and define, for each overlap spline
$s=\{p_i\}_{i=1}^{N-1}\in S(X,\cG,\cP)$ the continuous function $\tilde s$ on $[a,b-h]$ given by
$$
\tilde s|_{[t_{i-1},t_i]}=p_i,\quad i=1,\ldots,N-1,$$
then we arrive at a special case of  \textit{discrete splines}, as defined for example in \cite[Section 8.5]{SchumSplines07}.
Note that the discrepancy that the discrete splines are defined on the interval $[a,b-h]$ instead of $G=[a,b]$ is technical and
is related to the fact that discrete splines use forward differences in the definition of the connection conditions, which
causes some asymmetry.

Clearly, various results known for discrete splines, such as dimension formulas, local bases (discrete B-splines) or
quasi-interpolation methods, see \cite{SchumSplines07} and references therein, 
can be carried over to appropriate more general types of univariate overlap splines.

\subsubsection*{Multivariate Polynomial Overlap Splines}

Let $G\subset\RR^d$, $d\ge 2$, $X=\{x_j\}_{j=1}^N\subset G$ and, for some choice of $\cG=\{G_i\}_{i=1}^m$, $G_i\subset G$,
 let $\cP_q:=\{\Pi^d_q|_{G_i}\}_{i=1}^m$ for some $q\ge1$, where $\Pi^d_q$ denotes the space of
$d$-variate polynomials of total degree less or equal $q$. Then
$S(X,\cG,\cP_q)$ is a space of multivariate polynomial overlap splines.
To generate interpolatory overlap splines, $X_i$ must satisfy $|X_i|=\dim \Pi^d_q|_{X_i}=\dim \Pi^d_q|_{G_i}$. 

Given a finite set $X\subset\RR^d$, suppose that $\Delta=\{T_1,\ldots,T_m\}$ is a \textit{triangulation} of $X$, where each $T_i\in\Delta$
is a (closed) non-degenerate simplex with vertices in $X$, and $T_i\cap T_j$ is a common face of $T_i,T_j$ as soon as
$i\ne j$. 
One way to define overlap splines is by choosing $G=\cup_{i=1}^m T_m$ and
$\cG=\{T_i:i=1,\ldots,m\}$. Then each $X_i=X\cap T_i$ consists  of $d+1=\dim \Pi^d_1=\dim \Pi^d_1|_{T_i}$
vertices  of $T_i$, and  we obtain interpolatory overlap splines  $s=\{p_i\}_{i=1}^m\in S(X,\cG,\cP_1)$
that may be realized as continuous piecewise linear splines  $\tilde s$ %
by setting $\tilde s|_{T_i}=p_i$, $i=1,\ldots,m$. 
Thus, in this way we obtain the classical linear Courant elements. Instead of setting $G_i=T_i$, we may choose $G_i$ to be 
some simply connected sets enclosing $T_i$, such that $G_i\cap X=X_i$ still holds, leading to essentially
the same space of overlap splines, without making use of a partition of $G$.

Starting with the same $X$, but choosing different subsets $G_i$, we obtain completely different spaces of overlap splines.
As an illustration, suppose in the bivariate case that a triangulation $\Delta'=\{T'_1,\ldots,T'_m\}$ is chosen such that
the set of middle points of edges of all triangles of $\Delta'$  coincides with $X$. Then it is easy to see that
by choosing $G'_i$ as the triangles $T'_i$ or some sets $G'_i\subset\RR^2$ enclosing $T'_i$, 
such that $G'_i\cap X=T'_i\cap X$, $i=1,\ldots,m$,  
we obtain a space of overlap splines $S(X,\cG',\cP_1)$ that can be interpreted as the  discontinuous  
Crouzeix-Raviart finite elements, see e.g.\ \cite{BrennerScott08}.

If we start with a triangulation $\Delta=\{T_1,\ldots,T_m\}$, and define $X$ as its set of \textit{domain points} of degree
$q$, that is the union of the domain points for all simplices $T_i$ (see the definitions in \cite{LaiSchumaker07} for $d=2,3$,
and for example in \cite{AAD11} in general),
then by choosing $X_i=X\cap T_i$ and some enclosing sets $G_i\supset T_i$, with $G_i\cap X= X_i$, we get a space of
interpolatory overlap splines that may be identified with continuous piecewise polynomials of degree $q$ on the 
triangulation $\Delta$  %
in the same way as described before for $q=1$.

Thus, we obtained the classical spline spaces $S^0_q(\Delta)$ on triangulations, where 
$$
S^r_q(\Delta):=\big\{f\in C^r(\cup_{i=1}^mT_i):f|_{T_i}\in \Pi^d_q\big\}.$$ 
In order to interpret $S^r_q(\Delta)$ with $r\ge1$ as overlap splines we would need to generalize 
the patch connection conditions $p_i|_{X_i\cap X_j}=p_j|_{X_i\cap X_j}$ of Definition~\ref{osp} by either requiring
that the partial derivatives of $p_i$ and $p_j$ up to order $r$ coincide at $X_i\cap X_j$, or that the differences
$p_i-p_j$ belong to appropriate polynomial ideals.

In general no triangulation is needed; in order to define overlap splines we just choose $X\subset G$ and overlapping sets
$G_i$. Let us assume that $G$ is the closure of a domain $\Omega\subset\RR^d$ (that is, $\Omega$ an open and connected set),
choose a collection of domains $\Omega_i\subset\RR^d$, $i=1,\ldots,m$, that comprise an open cover of $G$, 
and set $G_i=\Omega_i\cap G$.  For any $X=\{x_j\}_{j=1}^N\subset G$ and any $q\in\NN$ we obtain a space of overlap splines
$S(X,\cG,\cP_q)$. 
We may also prefer to place some further restrictions on the sets $X_i$, for example requiring
that $\Omega_i$ are simply connected or convex.

Similar to the case of classical spline spaces (see \cite[Section 9.8]{LaiSchumaker07}),
 we may speak of a generic position of 
the set $X$.

\begin{definition}[Generic position of $X$]\label{generic} 
We say that $X=\{x_j\}_{j=1}^N$ is in a \textit{generic position} with respect to $\cG$
and $q$ if 
$$
\dim S(X,\cG,\cP_q)\le \dim S(\tilde X,\cG,\cP_q)$$
for all sets $\tilde X\subset G$ obtained by a sufficiently small perturbation
of the coordinates of the points $x_j$, $j=1,\ldots,N$.
\end{definition}

The following statement gives a justification for this definition, as it shows that small perturbations
of $X$ do not change the dimension of $S(X,\cG,\cP_q)$ if $X$ is in a generic position.

\begin{proposition}\label{genineq}
Given any $X,\cG,q$, we have $\dim S(X,\cG,\cP_q)\ge \dim S(\tilde X,\cG,\cP_q)$ for all sets
$\tilde X=\{\tilde x_j\}_{j=1}^N\subset G$ obtained by a sufficiently small perturbation
of the coordinates of the points $x_j$, $j=1,\ldots,N$.
\end{proposition}

\begin{proof}
Assume that the perturbation is so small that $\tilde J_i=J_i$, $i=1,\ldots,m$, where $J_i,\tilde J_i$ are index sets such that
$X_i:=X\cap G_i=\{x_j:j\in J_i\}$ and $\tilde X_i:=\tilde X\cap G_i=\{\tilde x_j:j\in \tilde J_i\}$. Then 
any $s=\{p_i\}_{i=1}^m\in S(\tilde X,\cG,\cP_q)$ may be uniquely represented by a sequence $c$ of $m{q+d\choose d}$ coefficients of $m$
polynomials $p_i\in\Pi^d_q$ satisfying the equation $A(\tilde X)c=0$, where $A(\tilde X)$ is a matrix whose rows are generated
from all connection conditions $p_i|_{\tilde X_i\cap \tilde X_j}=p_j|_{\tilde X_i\cap \tilde X_j}$ expressed in terms of the
coefficients in $c$. Hence
$$
\dim S(\tilde X,\cG,\cP_q)=\dim\ker A(\tilde X)  = m{q+d\choose d}-\rank A(\tilde X).$$
Arguing as in the proof of \cite[Theorem 9.32]{LaiSchumaker07}, we choose a nonsingular square submatrix $B$ of $A(X)$ 
of size $r=\rank A(X)$, and consider the submatrix $B(\tilde X)$ of $A(\tilde X)$ corresponding to the same rows and columns.
For a sufficiently small perturbation $\det B(\tilde X)\ne 0$ since $\det B(\tilde X)$ is a continuous function of
$\tilde X$ and $\det B(X)=\det B\ne 0$. Hence $\rank A(\tilde X)\ge r$, and the claim follows. 
\end{proof}

\runinhead{Open Problem (Dimension of the spaces of polynomial overlap splines)} 
In the above setting, determine the dimension of $S(X,\cG,\cP_q)$, in particular for the case
when  $X$ is in a generic position.
Characterize overlap spline spaces $S(X,\cG,\cP_q)$ for which the equality holds 
in \eqref{lbdim}, that is
\begin{equation}\label{lbdimeq}
\dim S(X,\cG,\cP_q)=|X| + m{q+d\choose d}-\sum_{i=1}^m|X_i|.
\end{equation}

\smallskip

Note that  \eqref{lbdimeq} holds in particular for interpolatory overlap splines.
Clearly, $S(X,\cG,\cP_q)$ is interpolatory, with $\dim S(X,\cG,\cP_q)=|X|$, if and only if each $X_i$ is an I-set for 
$\Pi^d_q$, as in the above examples on triangulations. 
This means that $|X_i|=\dim \Pi^d_q={q+d\choose d}$ and $X_i$ does not lie on any algebraic hypersurface of 
degree $q$. 

If we only assume that $|X_i|=\dim \Pi^d_q={q+d\choose d}$ for all $i=1,\ldots,m$, then all $X_i$ may be turned into
I-sets for $\Pi^d_q$ by a single arbitrarily small perturbation of $X=\{x_j\}_{j=1}^N$. Indeed, given such an $X$, 
consider the set $\fX$  consisting of all $N$-tuples 
$(\tilde x_1,\ldots,\tilde x_N)\in G^N$ such that 
$$
\text{
(a)\; $\tilde x_j\ne \tilde x_k$ whenever $j\ne k$, and\quad
(b)\; $\tilde x_j\in G_k$ if and only if $j\in J_k$,
}$$
 with $J_k$ defined as in the proof of Proposition~\ref{genineq}. 
For each $i=1,\ldots,m$ denote by $\fX_i$ the subset of  $\fX$ that consists of all 
$(\tilde x_1,\ldots,\tilde x_N)\in \fX$ such that
$$
\text{$\{\tilde x_j:j\in J_i\}$ is an I-set for $\Pi^d_q$.}$$
Then each $\fX_i$ is an open and dense subset of $\fX$, and hence $\tilde\fX:=\bigcap_{i=1}^m\fX_i$ is also an open dense
subset of $\fX$. Since $(x_1,\ldots,x_N)\in \fX$, by an arbitrary small perturbation of the coordinates of the points $x_j$,
$j=1,\ldots,N$, we may obtain a set $\tilde X=\{\tilde x_j\}_{j=1}^N$ such that 
$(\tilde x_1,\ldots,\tilde x_N)\in \tilde \fX$ and hence each $\tilde X_i=\tilde X\cap G_i$ is an I-set for $\Pi^d_q$,
and thus $S(\tilde X,\cG,\cP_q)$ is interpolatory with $\dim S(\tilde X,\cG,\cP_q)=|\tilde X|=|X|$.
This implies the following statement.

\begin{proposition}\label{geninterp} Assume that $|X_i|=\dim \Pi^d_q$ for all $i=1,\ldots,m$. Then $X$ is in a generic
position if and only if \eqref{lbdimeq} holds, that is  $\dim S(X,\cG,\cP_q)=|X|$. In particular, $X$ is in a generic
position if $S(X,\cG,\cP_q)$ is interpolatory.
\end{proposition}

\begin{proof} 
If $\dim S(X,\cG,\cP_q)=|X|$, then $X$ is in a generic position because the inequality 
$\dim S(\tilde X,\cG,\cP_q)\ge |\tilde X|=|X|$ for
sufficiently small perturbations follows from \eqref{lbdim}. Conversely, if $X$ is in a generic position and  
$\dim S(X,\cG,\cP_q)>|X|$, then any sufficiently small perturbation $\tilde X$ must satisfy 
$\dim S(\tilde X,\cG,\cP_q)>|X|=|\tilde X|$, which contradicts the above argument that we may always find $\tilde X$ 
with interpolatory  $S(\tilde X,\cG,\cP_q)$.
\end{proof}

Note that the converse of the last claim of Proposition~\ref{geninterp} is not true. For example, let $d=2$, $q=1$, $m=1$, 
let $G$ be any open set, and let $X$ consist of three collinear points in $G_1=G$. Then $S(X,\cG,\cP_1)$ is not interpolatory,
but it is in a generic position since $\dim S(X,\cG,\cP_1)=3=|X|$.

It follows from \cite[Theorem 8.50]{SchumSplines07} that \eqref{lbdimeq} holds for all spaces of
univariate discrete splines of \cite[Section 8.5]{SchumSplines07} interpreted as overlap splines in the same way as done
in the example defined by \eqref{ex1d}.

However, it is easy to generate an example where \eqref{lbdimeq} fails. Let $d=2$, $q=1$, $m=2$, 
let $\Omega_1$ and $\Omega_2$ be two connected open sets with nonempty intersection, and $\Omega=\Omega_1\cup \Omega_2$. 
We choose
$X$ to be a set of $N$ points in $\Omega_1\cap \Omega_2$. Then $X_1=X_2=X$ and the right hand side of \eqref{lbdimeq}
is $6-N$, which is even negative for $N>6$. It is easy to see that the dimension of $S(X,\cG,\cP_1)$ is three if the points in $X$ 
are non-collinear, four if they are collinear and $N\ge2$, and
five if $N=1$. Thus, \eqref{lbdimeq} is only correct if either $N\le2$, or $N=3$ and the points are non-collinear. 
Note that $X$ is in a generic position unless $N\ge3$ and the points are collinear, since otherwise small perturbations of $X$ do not
change the dimension of $S(X,\cG,\cP_1)$.

\subsubsection*{Kernel-Based Overlap Splines}

Recall that a function $K:G\times G\to R$ is said to be a \textit{positive definite kernel} if the matrix 
$[K(x_i,x_j)]_{i,j=1}^n$ is positive definite for any finite set $X=\{x_j\}_{j=1}^n\subset G$. Then $X$ is an I-set
for the space 
$$
P_{K,X}:=\spann\{K(\cdot,x_j):j=1,\ldots,n\}.$$ 
Therefore, an interpolatory overlap spline space 
$S(X,\cG,\cP)$ is obtained for any $X\subset G$ and $\cG=\{G_j\}_{j=1}^m$ with usual assumptions
if we choose $\cP=\cP_K:=\{P_{K,X\cap G_i}\}_{i=1}^m$.

A kernel $K$ is \textit{conditionally positive definite} with respect to a finite-dimensional space $Q$ of real-valued
functions on $G$ if for any finite set $X=\{x_j\}_{j=1}^n\subset G$ the quadratic form $\sum_{i,j=1}^n c_ic_jK(x_i,x_j)$ is
positive for all $c\in\RR^n\setminus\{0\}$ such that $\sum_{j=1}^n c_jp(x_j)=0$ for all $p\in Q$. If $X$ is a total set
for $Q$, then $X$ is an I-set for the space 
$$
P_{K,X,Q}:=\Big\{\sum_{j=1}^n c_j K(\cdot,x_j)+\tilde p:\tilde p\in Q\text{ and } 
c\in\RR^n\text{ with }\sum_{j=1}^n c_jp(x_j)=0\;\forall p\in Q\Big\},$$
see for example \cite{Wendland}. Note that $P_{K,X,\{0\}}=P_{K,X}$. For any finite $X\subset G$ and $\cG=\{G_j\}_{j=1}^m$ we obtain an interpolatory overlap spline 
space $S(X,\cG,\cP_{K,Q})$, with $\cP_{K,Q}:=\{P_{K,X\cap G_i,Q}\}_{i=1}^m$, as soon as each $X_i=X\cap G_i$ is total
for $Q$. If $G\subset\RR^d$ and $Q=\Pi^d_{q-1}$ for some $q\ge1$, then $K$ is said to be \textit{conditionally positive
definite of order $q$}, and we use the notation $P_{K,X,q}:=P_{K,X,Q}$ and  $\cP_{K,q}:=\cP_{K,Q}$. We also set 
 $P_{K,X,0}:=P_{K,X}$ and  $\cP_{K,0}:=\cP_{K}$ to include the positive definite case in the common notation.

We refer to \cite{Buhmann03,Fasshauer07,Wendland} for the theory of (conditionally) positive definite kernels. In particular,
for $G\subset\RR^d$  several families of kernels are known in the form of a \textit{radial basis function (RBF)}
$K(x,y)=\varphi(\|x-y\|_2)$. A typical example of a positive definite RBF is the \textit{Gauss kernel}
given by $\varphi(r)=e^{-(\eps r)^2}$, where $\eps>0$ is the \textit{shape parameter}. The  \textit{polyharmonic RBF} given
by  $\varphi(r)=r^\alpha$ for any $\alpha\in (0,+\infty)\setminus 2\NN$ is conditionally positive
definite of order $q\ge \lfloor\alpha/2\rfloor +1$.

Note that we can also choose different kernels $K_i$ and different spaces $Q_i$ for each $i=1,\ldots,m$. In particular,
$K_i$ may be the Gauss kernel with different values of the shape parameter.

\section{Approximation of Functions}%
\label{pum}

An overlap spline $s=\{p_i\}_{i=1}^m$, obtained by some computation as an approximation of 
a function $f:G\to\RR$, may be used directly to estimate function values $f(x)$ or other functionals such as partial
derivatives of $f$ at $x\in G$,
by evaluating at $x$ suitable patches $p_i$
such that $x\in G_i$. Depending on the application it may be sufficient to just pick one such $i$ and use $p_i(x)$, or we may need to build a
(weighted) average of all  values $p_i(x)$ with $x\in G_i$, or produce an approximating function $\tilde s:G\to\RR$ in a different way.
Moreover, the restriction $s|_X$ may already serve as a discrete approximation of $f$.

In special circumstances, like univariate generalized splines or continuous multivariate polynomial splines on triangulations,
the values $p_i(x)$ may be the same for all $i$ with $x\in G_i$, at least for a special choice of $G_i$, and we
obtain a well-defined function $\tilde s$ by using these values. However, this is a rare exception if we look for meshless
methods, because it typically means that 
$\cG=\{G_i\}_{i=1}^m$ is a partition of $G$.

We always have at our disposal the \textit{Partition of Unity Method (PUM)} that generates $\tilde s(x)$ as a weighted average of
$p_i(x)$ as follows. Let $\Gamma=\{\gamma_i\}_{i=1}^m$ be a \textit{partition of unity} associated with $\cG$, such that
$$
\gamma_i: G\to\RR,\quad \gamma_i(x)\ge 0,\; x\in G_i,\quad \gamma_i(x)= 0,\; x\in G\setminus G_i,\quad i=1,\ldots,m,$$
and
$$
\sum_{i=1}^m \gamma_i(x)=1,\quad x\in G.$$
For each overlap spline $s=\{p_i\}_{i=1}^N\in S(X,\cG,\cP)$ we define a function
$$
s_\Gamma(x)=\sum_{\substack{i=1\\ x\in G_i}}^m \gamma_i(x) p_i(x),\quad x\in G,$$
that can take the role of $\tilde s$. If $S(X,\cG,\cP)$ is interpolatory, then $s_\Gamma|_X=s|_X$.

Note that the patch connections $p_i|_{X_i\cap X_j}=p_j|_{X_i\cap X_j}$ of Definition~\ref{osp} are not needed in order that
$s_\Gamma$ can serve as an approximation of a given function $f:G\to\RR$. It suffices to find $p_i\in P_i$ that approximate $f$ well
on $G_i$, see e.g.\ the approximation theory of PUM in \cite{Wendland}. For example, for the approximation of
functions from scattered data $f(x_i)$, $i=1,\ldots,N$, with $x_i$ in a smooth manifold $G$, we may use a kernel-based 
interpolatory overlap spline 
$s\in S(X,\cG,\cP)$ and generate $s_\Gamma$ with desired degree of smoothness by using appropriately smooth partition of unity
functions $\gamma_i$. However, if we build a \textit{disconnected overlap spline} 
$s\in S(\emptyset,\cG,\cP)$, with local approximations $p_i\in P_i$ of the data
$f|_{X\cap G_i}$ obtained by any suitable method, then $s_\Gamma$ will also inherit the approximation quality of the patches.

The above \textit{direct approximation} setting, where the input data is given at $x_i\in X$, should be contrasted with 
the \textit{solution of the operator equations}, which we will consider in the next section.
In this case patch connections are crucial, 
and the discretized solutions $\hat u =s|_X$ may in fact be sufficient as the output of the algorithm,
without creating a solution $\tilde s$ defined on $G$, or one may leave the latter to a post-processing step with its own
algorithms that may in turn rely on PUM or other approaches to scattered data fitting, such as  kernel-based and moving least squares
methods or spline fitting.

\section{Collocation, Least Squares and Finite Difference Methods}
\label{col}

We now consider a linear operator equation: Find $u:G\to\RR$, such that 
\begin{equation}\label{opeq}
Lu(x)=f(x),\qquad \forall x\in G,
\end{equation}
where $f$ is a given function $f:G\to\RR$, and $Lu:G\to\RR$ is well defined for all $u$ in a vector space $U$ of functions
defined on $G$. We assume that the operator $L$ is linear and that \eqref{opeq} has a unique solution $u\in U$ for all 
$f\in F$, where $F$ is some other vector space of functions on $G$.  

For each $x\in G$, a subset $G_x\subset G$ such that $x\in G_x$ is said
to be a \textit{determining set of $L$ at $x$} if $Lu(x)=Lv(x)$ whenever $u,v\in U$ satisfy $u|_{G_x}=v|_{G_x}$. 
An overlap spline space $S(X,\cG,\cP)$ is \textit{compatible} with $L$ if $P_i\subset U|_{G_i}$, $i=1,\ldots,m$, and
every $G_i\in\cG$ is a determining set of $L$ at each $x\in G_i$. Then $Lp(x)$ is well defined for all 
$x\in G_i$, $p\in P_i$, and hence $Ls=\{Lp_i\}_{i=1}^m\in S(X,\cG,L\cP)$ is well defined for each 
$s=\{p_i\}_{i=1}^m\in S(X,\cG,\cP)$, where  $L\cP:=\{LP_i\}_{i=1}^m$, with $LP_i:=\{Lp: p\in P_i\}$.

The \textit{Method of Collocation} determines an approximate solution  of \eqref{opeq} as an overlap spline $s$ as follows.
Choose a space $S(X,\cG,\cP)$ of overlap splines compatible with $L$, a set of \textit{collocation nodes}
$Y=\{y_1,\ldots,y_D\}\subset G$, with $D=\dim S(X,\cG,\cP)$, and a mapping $\sigma:\{1,\ldots,D\}\to\{1,\ldots,m\}$,
such that  $y_j\in G_{\sigma(j)}$. Determine $s=\{p_i\}_{i=1}^m\in S(X,\cG,\cP)$  by requiring that
\begin{equation}\label{coleq}
Lp_{\sigma(j)}(y_j)=f(y_j),\qquad j=1,\ldots,D.
\end{equation}
If we expand the solution $s$ of \eqref{coleq} in a basis $\{s_1,\ldots,s_D\}$ of $S(X,\cG,\cP)$, 
\begin{equation}\label{ck}
s=\sum_{k=1}^Dc_ks_k,\qquad c_k\in\RR,
\end{equation}
where 
$$
s_k=\{p_{ki}\}_{i=1}^m\in S(X,\cG,\cP), \quad k=1,\ldots,D,$$
then the unknown coefficients $c_k$ of $s$ satisfy the square system of linear equations
\begin{equation}\label{colsys}
\sum_{k=1}^Dc_kLp_{k,\sigma(j)}(y_j)=f(y_j),\qquad j=1,\ldots,D.
\end{equation}
Although in general there is no guarantee that this linear system is non-singular or solvable, we will soon discuss 
some situations where this has been shown.

The \textit{Method of Discrete Least Squares} differs from the method of collocation in that the number of collocation nodes
exceeds the dimension of the space of overlap splines. Moreover, we allow multiple collocation nodes if different patches 
$p_{\sigma(j)}$ are chosen for the same node $y_j$. The exact fulfillment of the  equations \eqref{coleq} is 
replaced by the least squares minimization of the residual. 
More precisely, we choose a sequence $Y=(y_1,\ldots,y_M)$ of nodes in $G$, with $M>D=\dim S(X,\cG,\cP)$, 
where $y_k=y_j$ is allowed for $k\ne j$, and a mapping $\sigma:\{1,\ldots,M\}\to\{1,\ldots,m\}$,
such that  $y_j\in G_{\sigma(j)}$, and $\sigma(j)\ne\sigma(k)$ whenever $y_j=y_k$ with $j\ne k$.
Determine $s=\{p_i\}_{i=1}^m\in S(X,\cG,\cP)$  by solving the minimization problem
\begin{equation}\label{lseq}
\min\Big\{\sum_{j=1}^M|Lp_{\sigma(j)}(y_j)-f(y_j)|^2:\{p_i\}_{i=1}^m\in S(X,\cG,\cP)\Big\}.
\end{equation}
The algebraic formulation is 
\begin{equation}\label{lseqa}
\min\Big\{\|Ac-b\|_2:c\in\RR^D\Big\},
\end{equation}
where 
$$
A=[Lp_{k,\sigma(j)}(y_j)]_{j=1,k=1}^{M,D},\quad b=[f(y_j)]_{j=1}^{M},$$
and $c=[c_k]_{k=1}^{D}$ is the vector of the coefficients of $s$ in \eqref{ck}.
The least squares problem is always solvable. The solution is unique if and only if the system matrix $A$ is of full rank.

As alternative to the formulation \eqref{lseq}, also other approaches to the approximate solution of the overdetermined 
linear system
\begin{equation}\label{overdet}
Lp_{\sigma(j)}(y_j)=f(y_j),\quad j=1,\ldots,M,
\end{equation}
are possible, for example, weighted least squares.

\medskip

We now discuss several examples emphasizing connections to the Finite Difference Method and its meshless generalizations. 

\subsubsection*{A Boundary Value Problem on an Interval}

As a first example we consider the boundary value problem
\begin{equation}\label{ex1dBVP}
u''(x)=f(x),\quad x\in (a,b),\qquad u(a)=u(b)=0,
\end{equation}
with unknown $u\in U= C^2[a,b]$. Note that the operator $L$ of \eqref{opeq} takes the form
$$
Lu(x)=\begin{cases} 
u''(x),& x\in (a,b),\\
\hspace{6pt} u(x), & x\in \{a,b\},
\end{cases}$$
and the right hand side is extended to the endpoints of $[a,b]$ by setting $f(a)=f(b)=0$.

Let us apply the method of collocation using the overlap spline space $S(X,\cG,\cP)$ defined by \eqref{ex1d}. It is easy to
see that this space is compatible with the operator $L$. The dimension of $S(X,\cG,\cP)$ is
$D=N+1$, and we may look for the overlap spline solution $s=\{p_i\}_{i=1}^{N-1}$ in the form 
$$
s=\sum_{k=0}^{N}c_k\ell_k,\qquad c_k\in\RR,$$
using the Lagrange basis \eqref{ex1dL}. As discussed in Section~\ref{os}, $c_k$ can be interpreted as the values $s(x_k)$
of the overlap spline $s$ because the values of all patches $p_i(x_k)$ such that $x_k\in G_i$ coincide. We choose the
collocation nodes $y_j=x_j$, $j=0,\ldots,N$, and corresponding patches $p_{\sigma(j)}$ according to the rule
$$
\sigma(j)=j\text{ for }j=1,\ldots,N-1,\quad \sigma(0)=1,\quad \sigma(N)=N-1.$$
The entries $L\ell_{k,\sigma(j)}(x_j)$ of the system matrix of \eqref{colsys} are easy to compute using the explicit formulas 
\eqref{ex1dL}, and we arrive at the linear system
\begin{align}\label{ex1dBVPsys}
\begin{split}
\frac{c_{k-1}-2c_{k}+c_{k+1}}{h^2} &= f(x_k),\quad k=1,\ldots,N-1,\\
c_0 &=0,\quad c_N=0.
\end{split}
\end{align}
This is in fact the same non-singular linear system that determines the classical finite difference solution of the boundary value problem
\eqref{ex1dBVP}, where $c_k$ are approximations of $u(x_k)$. In particular, the standard theory of finite difference methods
applies that provide error bounds for $c_k-u(x_k)$, and thus for $s(x_k)-u(x_k)$, $k=1,\ldots,N-1$.

A second look at the above reveals that the appearance of the second order finite difference 
$$
\frac{s(x_{k-1})-2s(x_{k})+s(x_{k+1})}{h^2},$$
that can be interpreted as the second derivative of the discrete function $s|_X$ at $x_k$, is not accidental because
the finite difference formula
$$
p''(x_k)=\frac{p(x_{k-1})-2p(x_{k})+p(x_{k+1})}{h^2},\quad p\in\Pi^1_2,$$
is just a method to compute $p''_k(x_k)$ using $p_k|_{X_k}$ when we set up the collocation system \eqref{coleq}
using the Lagrange basis of $S(X,\cG,\cP)$. Therefore, applying the method of collocation with interpolatory overlap splines 
in terms of their Lagrange basis may always be interpreted as a finite difference method.

Note that collocation with cubic discrete splines has been used in \cite{ChenWong14} to solve second order boundary 
value problems.

\subsubsection*{Dirichlet Problem for the Poisson Equation on the Square}

Consider the boundary value problem
\begin{equation}\label{exPoisson}
\Delta u(x)=f(x),\quad x\in \Omega=(0,1)^2,\qquad u|_{\partial\Omega}=0,
\end{equation}
where  $u\in U= C^2([0,1]^2)$. We define overlap splines by choosing $G=[0,1]^2$,
\begin{align}\label{ex2d}
\begin{split}
X&=\{x_{ij}\}_{i,j=0}^N,\;\text{ with } x_{ij}=(ih,jh)\in [0,1]^2,\; h=1/N,\\
\cG&=\{G_{ij}\}_{i,j=1}^{N-1},\;\text{ with }G_{ij} =\{x\in [0,1]^2:\|x-x_{ij}\|_2<h+\eps\},
\end{split}
\end{align}
where $0<\eps<(\sqrt{2}-1)h$. Then $X_{ij}=G_{ij}\cap X=\{x_{ij},x_{i-1,j},x_{i+1,j},x_{i,j-1},x_{i,j+1}\}$
is the 5-star of the node $x_{ij}$ as used in the standard finite difference discretization of \eqref{exPoisson}.

If we comprise $\cP$ of the polynomial spaces $P_{ij}=\Pi^2_2$ then the resulting space 
$S(X,\cG,\cP_2)$ is not interpolatory. Its dimension is nevertheless easy to determine by using
\eqref{kerimT}. Indeed, since $\dim \Pi^2_2|_{X_{ij}}=5=|X_{ij}|$, we conclude that $\dim\im T=|X|=(N+1)^2$. Further, 
$\ker T$ is generated by the overlap splines $s=\{p_{ij}\}_{i,j=1}^{N-1}$, with $p_{k\ell}(x_1,x_2)=(x_1-kh)(x_2-\ell h)$
for a single pair $(k,\ell)$, and $p_{ij}=0$ whenever $(i,j)\ne (k,\ell)$. There are $(N-1)^2$ overlap splines of this type,
hence $\dim\ker T=(N-1)^2$ and $D=\dim S(X,\cG,\cP_2)=(N+1)^2+(N-1)^2$, such that \eqref{lbdimeq} holds in this case. 
However, $\Delta s=\{\Delta p_{ij}\}_{i,j=1}^{N-1}\equiv0$ for all $s\in \ker T$, which implies that no choice of 
collocation nodes $y_1,\ldots,y_D$ would lead to a regular system matrix for \eqref{colsys}.

We now reduce $\Pi^2_2$ to a five-dimensional subspace complementary to $\ker T$, by taking
\begin{equation}\label{ex2d'}
\cP=\{P_{ij}\}_{i,j=1}^{N-1},\;\text{ with }P_{ij}=\spann\{1,x_1, x_2,x_1^2, x_2^2\}.
\end{equation}
 Then $S(X,\cG,\cP)$ is an interpolatory overlap spline space with dimension $(N+1)^2$. By using its Lagrange basis,
choosing collocation nodes $Y=X$, and evaluating $\Delta p_{ij}$ at $y_{ij}=x_{ij}$, we obtain for the patches $p_{ij}$
of any overlap spline $s=\{p_{ij}\}_{i,j=1}^{N-1}\in S(X,\cG,\cP)$ the identities
$$
\Delta p_{ij}(x_{ij})=\frac{1}{h^2}\Big(p_{ij}(x_{i-1,j})+p_{ij}(x_{i+1,j})+p_{ij}(x_{i,j-1})+p_{ij}(x_{i,j+1})
-4p_{ij}(x_{ij})\Big),$$
and see that the method of collocation in this setting is equivalent to the finite difference method for \eqref{exPoisson}
with the standard five point stencil.
Note that $P_{ij}$ of \eqref{ex2d'} are the spaces of \textit{least interpolation} \cite{deBoorRon92} for the corresponding 
sets $X_{ij}$.

As an alternative to reducing the polynomial spaces, a somewhat artificial remedy is to extend each set $X_{ij}$
by one point $\tilde x_{ij}$ in $G_{ij}$ not lying in any other $G_{k\ell}$ or on the lines $x_1=ih$ and $x_2=j h$.
Then $S(\tilde X,\cG,\cP_2)$ with $\tilde X=X\cup \{\tilde x_{ij}\}_{i,j=1}^{N-1}$ is interpolatory. Since the Lagrange 
polynomial corresponding to the point $\tilde x_{ij}$ is a multiple of $(x_1-ih)(x_2-jh)$, its Laplacian vanishes at 
$x_{ij}$, and we arrive at the same five point stencil.

\subsubsection*{Meshless Finite Difference Methods}

Overlap spline collocation or discrete least squares for the problem \eqref{opeq} with interpolatory spaces $S(X,\cG,\cP)$
is equivalent to a meshless finite difference method defined as follows. 

\begin{definition}[Meshless Finite Difference Method]\label{mFD}
Given a set $X=\{x_i\}_{i=1}^N\subset G$ of \textit{discretization nodes} and a sequence $Y=(y_j)_{j=1}^M\subset G$ of 
\textit{collocation nodes}, with $M\ge N$, choose for each $j=1,\ldots,M$ a \textit{set of influence} $X_j^\mathrm{ndf}\subset X$ and
a numerical differentiation formula (ndf)
\begin{equation}\label{ndf}
Lv(y_j)\approx\sum_{x_i\in X_j^\mathrm{ndf}}w_{ji} v(x_i),\quad v\in U,
\end{equation}
defined by some weights $w_{ji}\in \RR$. Determine the vector $\hat u\in\RR^N$ from the linear system
\begin{equation}\label{mFDsys}
\sum_{x_i\in X_j^\mathrm{ndf}}w_{ji}\hat u_i = f(y_j), \quad j=1,\ldots,M,
\end{equation}
and use $\hat u_j$ as approximation of $u(x_j)$, $j=1,\ldots,N$, for the solution $u$ of \eqref{opeq}.
\end{definition}

If $M=N$ and the system \eqref{mFDsys} is nonsingular, then $\hat u$ is the unique solution of \eqref{mFDsys}. In the case
$M>N$ an appropriate method for overdetermined systems has to be applied, usually the method of least squares. We expect that
the system matrix of \eqref{mFDsys} is of full rank, which is normally observed in numerical experiments but only rarely
rigorously guaranteed by the design of particular algorithms.

Numerical differentiation formulas \eqref{ndf} are usually obtained by requiring exactness  for all elements of
appropriately chosen finite-dimensional spaces $P_j^\mathrm{ndf}$ of functions whose domain of definition contains $X_j^\mathrm{ndf}$,
\begin{equation}\label{exact}
Lp(y_j)=\sum_{x_i\in X_j^\mathrm{ndf}}w_{ji} p(x_i),\quad \forall p\in P_j^\mathrm{ndf}.
\end{equation}
this presumes, of course, that \eqref{exact} is solvable with respect to $w_{ji}$. If the conditions \eqref{exact} do not
determine the weights  $w_{ji}$ uniquely, then a particular solution needs to be chosen. 

Several variants of the meshless finite difference method have been proposed, depending on how the collocation nodes are
chosen in relation to $X$, and how the formulas \eqref{ndf} are obtained. 

It is easy to see that both the method of collocation and  the method of discrete least squares for an interpolatory overlap
spline space $S(X,\cG,\cP)$ with the Lagrange basis $\ell_k=\{\ell_{ki}\}_{i=1}^m$, $k=1,\ldots,N$, is equivalent to a version
of the meshless finite difference method, where the  $j$th set of influence is the I-set for  
$P_j^\mathrm{ndf}:=P_{\sigma(j)}$
given by $X_j^\mathrm{ndf}:=X_{\sigma(j)}=X\cap G_{\sigma(j)}$,  and the weights of \eqref{ndf} are obtained as
$w_{jk}=L\ell_{k,\sigma(j)}(y_j)$, or, equivalently, as the unique solution  of \eqref{exact}. Then in the case of the method
of collocation the systems \eqref{colsys} and  \eqref{mFDsys} are the same, and hence their solution vectors $c$,
respectively $\hat u$, are identical. In the case of the method of discrete least squares the vector  $c$ of \eqref{lseqa}
gives a least squares solution of the overdetermined system \eqref{mFDsys}. 

Note that in the case when $Lu(y_j):=u(y_j)$ and $y_j\in X_j^\mathrm{ndf}$, the solution of \eqref{exact} is given by 
$w_{ji}=\delta_{ji}$, so that the corresponding equation in \eqref{mFDsys} becomes $\hat u_j=f(y_j)$, which is in particular
a natural way to handle Dirichlet boundary conditions. It does not matter in this case how we choose $X_j^\mathrm{ndf}$.

We now list some particular versions of the interpolatory overlap spline methods that can be interpreted as known
meshless finite difference methods. In all of them we assume that $G$ is a closure of a bounded domain in a smooth 
manifold $\cM$, and \eqref{opeq} is a Dirichlet boundary value problem as long as $G\ne\cM$, that is $Lu(x)=u(x)$ for all
$x\in \partial G$. 

\begin{enumerate}
\item The method of collocation for an interpolatory polynomial overlap spline space $S(X,\cG,\cP_q)$, with $\cM=\RR^d$,
$G=\{G_j\}_{j=1}^N$, where the sets $G_j$ are chosen such that each $X_j=X\cap G_j$, $j=1,\ldots,N$, is an I-set 
for $\Pi^d_q$. In particular, $X_j$ may consist of ${d+q\choose d}$ neighbors of $x_j$, including $x_j$ itself, and the
 collocation nodes and 
the mapping  $\sigma$ be given by $y_j=x_j$  and  $\sigma(j)=j$, for all $j=1,\ldots,N$. Note that, as explained above,
in this case we can also do without choosing $X_j$ for boundary nodes $x_j\in \partial G$ because the corresponding equations in 
\eqref{colsys} and  \eqref{mFDsys} are always $c_j=f(x_j)$, respectively, $\hat u_j=f(x_j)$.
This leads to a meshless finite difference method first considered in \cite{Jensen72}, where $X_j$ is chosen as the set of 
${d+q\choose d}$ nearest neighbors of $x_j$, and subsequently corrected should it fail to be an I-set. More recent approaches,
see for example \cite{D19arxiv,DavySaf21,Seibold08,ShenLvShen09} rely on the selection of $X_j$ by algorithms that generate I-sets 
or subsets of I-sets, such as in classical five point stencil discussed above for the Poisson equation, while trying 
to optimize the accuracy of the numerical  differentiation in \eqref{ndf} or guarantee the solvability of \eqref{mFDsys}.

\item The method of collocation for the polynomial least interpolation overlap spline space $S(X,\cG,\cP_\mathrm{least})$,
where $\cM=\RR^d$,  and each $P_j$ is the least interpolation space of polynomials
\cite{deBoorRon92} for the set $X_j=X\cap G_j$, $j=1,\ldots,N$. An example of the least interpolation overlap spline space
has  already been discussed above  for the Poisson equation solved by  the classical finite difference method with the five
point stencil. The meshless finite difference method that results from this setting in the case of $|X|=|\cG|$, with
$X_j$ being a set of neighbors of $y_j=x_j$  and  $\sigma(j)=j$, $j=1,\ldots,N$, has in fact been considered in 
\cite{DavyOanh11sp}, as a limiting case \cite{Schaback05} of the RBF-FD method with the Gauss kernel for its shape 
parameter going to zero. 

\item The method of collocation for the kernel-based overlap spline space $S(X,\cG,\cP_{K,Q})$, with a positive definite or
conditionally positive definite kernel $K$. The resulting meshless finite difference method, where $|X|=|\cG|$, 
$X_j$ a set of neighbors of $y_j=x_j$  and  $\sigma(j)=j$, $j=1,\ldots,N$, was first considered in 
\cite{TolShir03}. It is commonly called the RBF-FD method in the case when $K$ is a radial basis function, see 
\cite{FFprimer15} and references therein for more details. A recent overview of parameter choices  can be found in
\cite{LeBorneLeinen23}.  In the case of nontrivial $Q\ne\{0\}$ the sets $X_i$ do not have to be total for $Q$. It
suffices that each $X_j$, $j=1,\ldots,N$, is such that the exactness equations  $Lp(x_j)=\sum_{x_i\in X_j}w_{ji} p(x_i)$,
$p\in Q$, are solvable with respect to the weights $w_{ji}$, see \cite{D21}. Then the space $S(X,\cG,\cP_{K,Q})$ may not be
interpolatory, but $\dim P_j|_{X_j}=|X_j|$ holds for all $j$. Hence either an appropriate reduction of the spaces
$P_j=P_{K,X_j,Q}$ or an extension of $X_j$ for all $j$ with $\dim P_j>|X_j|$  will produce an interpolatory space of overlap
splines, similar to the five point stencil case discussed in detail above for the solution of the Poisson equation on the square
with polynomial overlap splines.

\item The method of discrete least squares for the kernel-based overlap spline space $S(X,\cG,\cP_{K,Q})$, with $|X|=|\cG|$, 
$X_j$ being a set of neighbors of $x_j$, $j=1,\ldots,N$. Choose collocation nodes $Y=\{y_i\}_{i=1}^M$ without repetitions, 
$M>N$, and let $\sigma(i)$ be the index of the node $x_j$ closest to $y_i$, that is 
$\|y_i-x_{\sigma(i)}\|=\min_{1\le j\le N}\|y_i-x_j\|$, for some distance in $G$, for example the Euclidean distance for 
$G\subset\RR^d$. This approach has been considered in \cite{TLH21}, which also provides an error analysis for the corresponding
meshless finite difference method. In an earlier paper \cite{LSH17} a similar method was applied after transforming
kernel-based overlap splines $s\in S(X,\cG,\cP_{K,Q})$ into functions $s_\Gamma$ created by a suitable partition of unity
$\Gamma$.

\item The method of discrete least squares for the kernel-based overlap spline space $S(X,\cG,\cP_{K,Q})$, where the sequence
$Y$ is obtained by combining all sets $X_i=\{x_j:j\in J_i\}$  and $\sigma(k)$ is the index $i$ of the set $J_i$ the node
$y_k=x_j$ comes with. In other words,
\begin{align*}
Y&=(x_j,j\in J_1,x_j,j\in J_2,\ldots,x_j,j\in J_m),\\
(\sigma(1),\ldots,\sigma(|Y|)&=(\underbrace{1,\ldots,1}_{|J_1|\text{ times}},\underbrace{2,\ldots,2}_{|J_2|\text{ times}},
\ldots,\underbrace{m,\ldots,m}_{|J_m|\text{ times}}).
\end{align*}
The resulting meshless finite difference method has been considered in \cite{D23}, with the full rank property and
error analysis provided under certain assumptions for elliptic equations \eqref{opeq} on closed manifolds. 

\end{enumerate}

\bibliographystyle{abbrv}

\bibliography{references} %

\end{document}